\documentclass[a4paper]{amsart}
\usepackage{graphicx}
\usepackage{amssymb}
\usepackage{amsmath}
\usepackage{amsthm,amsfonts,bbm}
\usepackage{amscd}
\usepackage[all,2cell]{xy}

\UseAllTwocells \SilentMatrices

\newtheorem{thm}{Theorem}[section]

\newtheorem{cor}[thm]{Corollary}
\newtheorem{lem}[thm]{Lemma}

\newtheorem{prop-def}[thm]{Proposition-Definition}
\theoremstyle{definition}
\newtheorem{defi}[thm]{Definition}
\theoremstyle{remark}

\newtheorem{exm}[thm]{\bf Example}
\numberwithin{equation}{section}
\numberwithin{figure}{section}

\newcommand{\A}{\mathcal{A}}

\def\la{\lambda}

\def\A{\mathcal{A}}

\def \S{\mathcal{S}}
\def\PS{\mathbb{S}}
\def\HPS{{\rm H}\mathbb{S}}

\def\diag{{\rm diag}}
\def\x{{\mathbf x}}
\def\y{{\mathbf y}}

\def\C{\mathbb{C}}

\def\Z{\mathbb{Z}}

\def\V{\mathcal{V}}
\def \PV{\mathbb{V}}
\def\HPV{{\rm H}\mathbb{V}}

\def\I{\mathcal{I}}
\def\B{\mathcal{B}}

\def\diag{{\rm diag}}
\def \i{\mathbf{i}}
\def \A{\mathcal{A}}
\def \L{\mathcal{L}}
\def \Q{\mathcal{Q}}

\def \D{\mathcal{D}}
\def \Spec{\mbox{\rm Spec}}
\def \HSpec{\mbox{\rm HSpec}}

\def\cl{\mbox{\rm cl}}

\begin{document}
\title[Eigenvector of Laplacian or signless Laplacian of Hypergraphs]
{Eigenvectors of Laplacian or signless Laplacian of Hypergraphs Associated with Zero Eigenvalue}

\author[Y.-Z. Fan]{Yi-Zheng Fan$^*$}
\address{School of Mathematical Sciences, Anhui University, Hefei 230601, P. R. China}
\email{fanyz@ahu.edu.cn}
\thanks{$^*$The corresponding author.
This work was supported by National Natural Science Foundation of China (Grant No. 11871073, 11771016, 11871071).}

\author[Y. Wang]{Yi Wang}
\address{School of Mathematical Sciences, Anhui University, Hefei 230601, P. R. China}
\email{wangy@ahu.edu.cn}

\author[Y.-H. Bao]{Yan-Hong Bao}
\address{School of Mathematical Sciences, Anhui University, Hefei 230601, P. R. China}
\email{baoyh@ahu.edu.cn}

\author[J.-C. Wan]{Jiang-Chao Wan}
\address{School of Mathematical Sciences, Anhui University, Hefei 230601, P. R. China}
\email{1500256209@qq.com}

\author[M. Li]{Min Li}
\address{School of Mathematical Sciences, Anhui University, Hefei 230601, P. R. China}
\email{1736808193@qq.com}

\author[Z. Zhu]{Zhu Zhu}
\address{School of Mathematical Sciences, Anhui University, Hefei 230601, P. R. China}
\email{2937242741@qq.com}


\subjclass[2000]{Primary 15A18, 05C65; Secondary 13P15, 14M99}



\keywords{Hypergraph, tensor, eigenvector, Laplacian, signless Laplacian, zero eigenvalue}

\begin{abstract}
Let $G$ be a connected $m$-uniform hypergraph.
In this paper we mainly consider the eigenvectors of the Laplacian or signless Laplacian tensor of $G$ associated with zero eigenvalue,
called the first Laplacian or signless Laplacian eigenvectors of $G$.
By means of the incidence matrix of $G$, the number of first Laplacian or signless Laplacian (or H-)eigenvectors can be obtained explicitly by solving
the Smith normal form of the incidence matrix over $\Z_m$ (or $\Z_2)$.
Consequently, we prove that the number of first Laplacian (H-)eigenvectors is equal to the number of
first signless Laplacian (H-)eigenvectors when zero is an (H-)eigenvalue of the signless Laplacian tensor.
We establish a connection between first Laplacian (signless Laplacian) H-eigenvectors and the even (odd) bipartitions of $G$.
\end{abstract}

\maketitle

\section{Introduction}
A {\it hypergraph} $G=(V(G),E(G))$ consists of a set of vertices, say $V(G)=\{v_1,v_2,\ldots,v_n\}$ and a set of edges, say $E(G)=\{e_{1},e_2,\ldots,e_{k}\} \subseteq 2^{V(G)}$.
If $|e_{j}|=m$ for each $j \in [k]:=\{1,2,\ldots,k\}$, then $G$ is called an {\it $m$-uniform} hypergraph.
In particular, the $2$-uniform hypergraphs are exactly the classical simple graphs.
One can refer \cite{Ber} for more on hypergraphs.

Recently, spectral hypergraph theory is proposed to explore connections between
the structure of a uniform hypergraph and the eigenvalues or eigenvectors of some related tensors.
Here the tensors may be called hypermatrices, which are multi-dimensional arrays of entries in some field,
  and can be viewed to be the coordinates of the classical tensors (as multilinear functions) under a certain (orthonormal) basis.
To be precise,
  a real {\it tensor} (or {\it hypermatrix}) $\A=(a_{i_{1}\ldots i_{m}})$ of order $m$ and dimension $n$ refers to a
  multi-dimensional array of entries $a_{i_{1}\ldots i_{m}} \in \mathbb{R}$ for all $i_{j}\in [n]$ and $j\in [m]$.
  The tensor $\A$ is called \textit{symmetric} if its entries are invariant under any permutation of their indices.

The eigenvalues of a tensor were introduced by Qi \cite{Qi, Qi2} and Lim \cite{Lim} independently.
To find the eigenvalues of a tensor, Qi \cite{Qi, Qi2} introduced the characteristic polynomial of a tensor,
 which is defined to be a resultant of a system of homogeneous polynomials.

 Given a vector $\x=(x_1,x_2,\ldots,x_n) \in \mathbb{C}^{n}$, $\A\x^{m-1} \in \C^n$, which is defined as follows:
  $$
  (\A\x^{m-1})_i =\sum_{i_{2},\ldots,i_{m}\in [n]}a_{ii_{2}\ldots i_{m}}x_{i_{2}}\cdots x_{i_m}, \mbox{~for~} i \in [n].
  $$
 Let $\mathcal{I}$ be the {\it identity tensor} of order $m$ and dimension $n$, that is, $i_{i_{1}i_2 \ldots i_{m}}=1$ if
   $i_{1}=i_2=\cdots=i_{m} \in [n]$ and $i_{i_{1}i_2 \ldots i_{m}}=0$ otherwise.

\begin{defi}{\em \cite{Lim,Qi}} Let $\A$ be an $m$-th order $n$-dimensional real tensor.
For some $\lambda \in \mathbb{C}$, if the polynomial system $(\lambda \mathcal{I}-\A)\x^{m-1}=0$, or equivalently $\A\x^{m-1}=\lambda \x^{[m-1]}$, has a solution $\x\in \mathbb{C}^{n}\backslash \{0\}$,
then $\lambda $ is called an \emph{eigenvalue} of $\A$ and $\x$ is an \emph{eigenvector} of $\A$ associated with $\lambda$,
where $\x^{[m-1]}:=(x_1^{m-1}, x_2^{m-1},\ldots,x_n^{m-1})$.
\end{defi}

If $\x$ is a real eigenvector or can be scaled into a real eigenvector of $\A$, surely the corresponding eigenvalue $\lambda$ is real.
In this case, $\lambda$ is called an {\it H-eigenvalue} of $\A$ and $\x$ is called an \emph{H-eigenvector}.
If $\x$ cannot be scaled into real, then $\x$ is called an \emph{N-eigenvector}.
If any eigenvector associated with $\la$ is an N-eigenvector, then $\la$ is called an {\it N-eigenvalue}.

The {\it characteristic polynomial} $\varphi_\A(\la)$ of $\A$ is defined to be the resultant of the system of polynomials $(\la \I-\A)\x^{m-1}$ in the indeterminant $\la$ \cite{CPZ2,Qi2}.
It is known that $\la$ is an eigenvalue of $\A$ if and only if it is a root of $\varphi_\A(\la)$.
The {\it spectrum} of $\A$, denoted by $\Spec(\A)$, is the multi-set of the roots of $\varphi_\A(\la)$, including multiplicity.
The largest modulus of the elements in $\Spec(\A)$ is called the \emph{spectral radius} of $\A$, denoted by $\rho(\A)$.
Denote by $\HSpec(\A)$ the set of distinct H-eigenvalues of $\A$.

Under the definitions of tensors and their eigenvalues, three main kinds of tensors associated with a uniform hypergraph were proposed, and their eigenvalues were
applied to investigate the structure of the hypergraph.
Cooper and Dutle \cite{CD} introduced the adjacency tensor of a uniform hypergraph, and
Qi \cite{Qi3} introduced the Laplacian tensor and signless Laplacian tensor.

\begin{defi}\cite{CD,Qi3}
Let $G=(V, E)$ be an $m$-uniform hypergraph on $n$ vertices.
The \emph{adjacency tensor} $\A(G)=(a_{i_1\cdots i_m})$ of $G$ is defined to be an $m$-th order and $n$-dimensional tensor whose entries are given by
\begin{align*}
a_{i_1\cdots i_m}=\begin{cases}
\dfrac{1}{(m-1)!}, & \mbox{if~} \{i_1,\cdots, i_m\}\in E, \\
0, & {\rm otherwise}.
\end{cases}
\end{align*}
Let $\D$ be an $m$-th order and $n$-dimensional diagonal tensor whose diagonal entries $d_{i\ldots i}$ is exactly the degree of the vertex $i$ for $i \in [n]$.
The \emph{Laplacian tensor} of $G$ is defined as $\L(G)=\D(G)-\A(G)$ and the \emph{signless Laplacian tensor} of $G$ is defined as $\Q(G)=\D(G)+\A(G)$.
\end{defi}

Surely, $\L(G)$ always has a zero eigenvalue with the all-ones vector $\mathbf{1}$ as an eigenvector.
When $G$ is connected, $\Q(G)$ has a zero eigenvalue if and only if $G$ is odd-colorable.
 For convenience, the eigenvectors of $\L(G)$ (respectively, $\Q(G)$) corresponding to the zero eigenvalue are called the \emph{first Laplacian eigenvectors}
(respectively, \emph{first signless Laplacian eigenvectors}) of $G$.
When such eigenvectors are an H (or N)-eigenvector, we will call them first Laplacian and first signless Laplacian H (or N)-eigenvectors respectively.

There are many results on the above tensors associated with hypergraph, see e.g. \cite{CD, CQ, FKT, HQ, HQS, HQX, KLQY, KF, KF2, LM, Ni,PZ, Qi3,SSW,ZSWB}.
Shao et al. \cite{SSW} used the equality of H-spectra of the Laplacian and signless Laplacian tensor to characterize the odd-bipartiteness of a hypergraph.
Nikiforov \cite{Ni} applied the symmetric spectrum of the adjacency tensor to characterize the odd-colorable property of a hypergraph.
Fan et al. \cite{FHB} characterized the spectral symmetry of the adjacency tensor of a hypergraph by means of the generalized traces which are related to the structure of the hypergraph.

In the paper \cite{FBH} Fan et al. proved that up to a scalar, there are a finite number of eigenvectors of the adjacency tensor associated with the spectral radius,
and such number can be obtained explicitly by the Smith normal form of incidence matrix of the hypergraph.
Hu and Qi \cite{HQ} established a connection between the number of first Laplacian (signless Laplacian) H-eigenvectors
 and the number of even (odd)-bipartite components of a uniform hypergraph.
They also discuss the number of first Laplacian (signless Laplacian) N-eigenvectors for $3$, $4$ or $5$-uniform hypergraphs by means of some kinds of partitions of the vertex set.
However, determining the number of even (odd)-bipartite components or the partitions as described in the paper is not so easy.

In this paper we will determine the number of the first Laplacian or signless Laplacian eigenvectors of a connected $m$-uniform hypergraph.
By means of the incidence matrix of a hypergraph, the number of first Laplacian or signless Laplacian (or H-)eigenvectors can be obtained explicitly by solving
the Smith normal form of the incidence matrix over $\Z_m$ (or $\Z_2$).
Consequently, we prove that the number of first Laplacian (H-)eigenvectors is equal to the number of
first signless Laplacian (H-)eigenvectors when zero is an (H-)eigenvalue of the signless Laplacian tensor.
We establish a connection between first Laplacian (signless Laplacian) H-eigenvectors and the even (odd) bipartitions of a hypergraph.

\section{Preliminaries}
Let $\A$ be an $m$-th order $n$-dimensional real tensor.
$\A$ is called \emph{reducible} if there exists a nonempty proper subset $I \subset [n]$ such that
$a_{i_{1}i_2\ldots i_{m}}=0$ for any $i_1 \in I$ and any $i_2,\ldots,i_m \notin I$;
if $\A$ is not reducible, then it is called \emph{irreducible} \cite{CPZ}.
We also can associate $\A$ with a directed graph $D(\A)$ on vertex set $[n]$ such that $(i,j)$ is an arc of $D(\A)$ if
and only if there exists a nonzero entry $a_{ii_2\ldots i_{m}}$ such that $j \in \{ i_2\ldots i_{m}\}$.
Then $\A$ is called \emph{weakly irreducible} if $D(\A)$ is strongly connected; otherwise it is called \emph{weakly reducible} \cite{FGH}.
It is known that if $\A$ is irreducible, then it is weakly irreducible; but the converse is not true.
The adjacency tensor, Laplacian or signless Laplacian tensor of a hypergraph is weakly irreducible if and only if the hypergraph is connected \cite{PZ, YY3}.

\subsection{Nonnegative tensors}
The Perron-Frobenius theorem was generalized from nonnegative matrices to nonnegative tensors by Chang et al. \cite{CPZ},
Yang and Yang \cite{YY1,YY2,YY3}, and Friedland et al. \cite{FGH}.
Here we list a part of the theorem.

According to the tensor product introduced by Shao \cite{Shao}, for a tensor $\A$ of order $m$ and dimension $n$, and two diagonal matrices $P,Q$ both of dimension $n$,
the product $P\A Q$ has the same order and dimension as $\A$, whose entries are given by
\begin{equation}\label{product}
(P\A Q)_{i_1i_2\ldots i_m}=p_{i_1i_1}a_{i_1i_2\ldots i_m}q_{i_2i_2}\ldots q_{i_mi_m}.
\end{equation}
If $P=Q^{-1}$, then $\A$ and $P^{m-1}\A Q$ are called \emph{diagonal similar}.
It is proved that two diagonal similar tensors have the same spectrum \cite{Shao}.

\begin{thm}\cite{YY3}\label{PF2}
Let $\A$ and $\B$ be $m$-th order $n$-dimensional tensors with $|\B|\le \A$. Then
\begin{enumerate}

\item[(1)] $\rho(\B)\le \rho(\A)$.

\item[(2)] If $\A$ is weakly irreducible and $\rho(\B)=\rho(\A)$, where $\la=\rho(\A)e^{\i\theta}$ is
an eigenvalue of $\B$ corresponding to an eigenvector $\y$, then $\y=(y_1, \cdots, y_n)$ contains no zero entries, and $\A=e^{-\i\theta}D^{-(m-1)}\B D$,
where $D=\diag(\frac{y_1}{|y_1|}, \cdots, \frac{y_n}{|y_n|})$.
\end{enumerate}
\end{thm}

\subsection{Affine and projective eigenvariety}
 Let $\la$ be an eigenvalue of a tensor $\A$ of order $m$ and dimension $n$.
 Let $\V_\la=\V_\la(\A)$ the set of all eigenvectors of $\A$ associated with $\la$ together with zero, i.e.
$$\V_\la(\A)=\{\x \in \mathbb{C}^n: \A\x^{m-1}=\la \x^{[m-1]}\}.$$
Observe that the system of equations $(\la \I-\A)\x^{m-1}=0$ is not linear yet for $m\ge 3$,
and therefore $\V_\la$ is not a linear subspace of $\mathbb{C}^n$ in general.
In fact, $\V_\la$ forms an affine variety in $\C^n$ \cite{Ha}, which is called the \emph{eigenvariety} of $\A$ associated with $\la$ \cite{HuYe}.
Let $\mathbb{P}^{n-1}$ be the standard complex projective spaces of dimension $n-1$.
Since each polynomial in the system $(\la \I-\A)\x^{m-1}=0$ is homogenous of degree $m-1$,
we consider the projective variety
$$\PV_\la=\PV_\la(\A)=\{\x \in \mathbb{P}^{n-1}: \A\x^{m-1}=\la \x^{[m-1]}\},$$
which is called the \emph{projective eigenvariety} of $\A$ associated with $\la$ \cite{FBH}.
In this paper the number of eigenvectors of $\A$ is considered in the projective eigenvariety $\PV_\la(\A)$, i.e.
the eigenvectors differing by a scalar is counted once as the same eigenvector.

\subsection{Incidence matrices}
Let $\A$ be a symmetric tensor of order $m$ and dimension $n$.
Set $$E(\A)=\{(i_1, i_2, \cdots, i_m)\in [n]^m: a_{i_1i_2\cdots i_m}\neq 0, 1\le i_1\le \cdots \le i_m \le n\}.$$
Define
\[b_{e,j}=|\{k: i_k=j, e=(i_1, i_2, \cdots, i_m) \in E(\A), k \in [m]\}|\]
and obtain an $|E(\A)|\times n$ matrix $B_\A=(b_{e,j})$, which is called the \emph{incidence matrix} of $\A$ \cite{FBH}.
Observe that $b_{ij}\in \{0, 1, \cdots, m-1, m\}$ so that the incidence matrix $B_\A$
can be viewed as one over $\Z_m$, where $\Z_m$ is the ring of integers modulo $m$.

The \emph{incidence matrix} of an $m$-uniform hypergraph $G$, denoted by $B_G=(b_{e,v})$, coincides with that of $\A(G)$, that is, for $e \in E(G)$ and $v \in V(G)$,
\begin{align*}
b_{e,v}=\begin{cases}
1, & \mbox{if~} \ v\in e,\\
0, & \textrm{otherwise}.
\end{cases}
\end{align*}

\subsection{$\Z_m$-modules}
Let $R$ be a ring with identity and $M$ a nonzero module over $R$. Recall that a finite chain of $l+1$ submodules of $M$
\[M=M_0>M_1>\cdots>M_{l-1}>M_l=0\]
is called a \emph{composition series} of length $l$ for $M$ provided that $M_{i-1}/M_i$ is simple for each $i \in [l]$.
By the Jordan-H\"older Theorem, if a module $M$ has a finite composition series,
then every pair of composition series for $M$ are equivalent.
The length of the composition series for such a module $M$ is called the \emph{composition length} of $M$,
denoted by $\cl(M)$; see \cite[Section 11]{AF} for more details.

We only focus on the modules over $\Z_m$. Let $M$ be a finitely generated module over $\Z_m$.
Suppose that $M$ is isomorphic to
\[\Z_{p_1^{k_1}}\oplus \Z_{p_2^{k_2}}\oplus \cdots \oplus \Z_{p_s^{k_s}},\]
where $p_i$ is a prime and $p_i^{k_i}| m$ for $i \in [s]$.
In this situation, $\cl(M)=\sum_{i=1}^s k_i$.
In particular, if $m$ is prime, then $\Z_m$ is a field, and $M$ is a linear space over $\Z_m$ whose dimension is exactly the composition length of $M$.
If $m=p_1^{k_1}p_2^{k_2}\cdots p_s^{k_s}$ with $p_i$'s being prime and distinct,
then as a regular module over $\Z_m$, $\cl(\Z_m)=\sum_{i=1}^s k_i=:\cl(m)$.
Also, if $d|m$, then $\Z_d$ is a submodule of the above regular module, and $\cl(\Z_d)=:\cl(d)$.

\subsection{Smith normal form of matrices over $\Z_m$}
Similar to the Smith normal form for a matrix over a principle ideal domain, one can also deduce a diagonal form for a matrix over $\Z_m$ by some elementary row transformations and column transformations, that is, for any matrix $B \in \Z_m^{k \times n}$,
there exist two invertible matrices $P \in \Z_m^{k \times k}$ and $Q \in \Z_m^{n \times n}$ such that
\begin{equation} \label{smith}
PBQ=\begin{pmatrix}
d_1 & 0 & 0 &  & \cdots & & 0\\
0 & d_2 & 0 &  & \cdots & &0\\
0 & 0 & \ddots &  &  & & 0\\
\vdots &  &  & d_r &  & & \vdots\\
 & & & & 0 & & \\
  & & & &  & \ddots & \\
0 &  &  & \cdots &  & &0
\end{pmatrix},
\end{equation}
where $r \ge 0$, $ 1 \le d_i \le m-1$, $d_i | d_{i+1}$ for $i=1,\ldots, r-1$, and $d_i |m$ for all $i=1,2,\ldots,r$.
The matrix in (\ref{smith}) is called the {\it Smith normal form} of $B$ over $\Z_m$,
$d_i$'s are the \emph{invariant divisors} of $B$.

Now consider a linear equation over $\Z_m$: $ B\x=0$, where $B \in \Z_m^{k \times n}$ and $\x \in \Z_m^n$.
The solution set
\[ \S=\{\x \in \Z_m^n:  B\x=0\}\]
is a $\Z_m$-module.
Suppose $B$ has a Smith normal form over $\Z_m$ as in (\ref{smith}).
Let $\S'=\{\x \in \Z_m^n: P B Q \x=0\}$.
Then $\x \in \S$ if and only if $Q^{-1} \x \in \S'$, implying that $\S$ is $\Z_m$-module isomorphic to $\S'$.
It is easy to see
$$\S \cong \S' \cong \oplus_{i, d_i \ne 1} \Z_{d_i} \oplus \underbrace{\Z_m \oplus \cdots \oplus \Z_m}_{n-r \text{\small~copies~}},$$
and the composition length of $\S$ is
$\sum_{i, d_i \ne 1} \cl(d_i)+ (n-r)\cl(m)$.

\section{First Laplacian or signless Laplacian eigenvectors}
We first introduce a property of the first Laplacian or signless Laplacian tensor of a connected uniform hypergraph given by Hu and Qi \cite{HQ}.

\begin{lem}\cite[Theorem 4.1]{HQ}\label{L&sL}
Let $G$ be an $m$-uniform connected hypergraph on $n$ vertices.

{\em (i)} $\x$ is a first Laplacian eigenvector of $G$ if and only if there exist a nonzero $\gamma \in \C$ and integers $\alpha_i$ such that
$x_i=\gamma e^{\i \frac{2\pi}{m} \alpha_i}$ for $i \in [n]$, and for each edge $e \in E(G)$,
\begin{equation}\label{eqL}\sum_{j \in e} \alpha_j=\sigma_e m,\end{equation}
for some integer $\sigma_e$ associated with $e$.

{\em (ii)} $\x$ is a first signless Laplacian eigenvector of $G$ if and only if there exist a nonzero $\gamma \in \C$ and integers $\alpha_i$ such that
$x_i=\gamma e^{\i \frac{2\pi}{m} \alpha_i}$ for $i \in [n]$, and for each edge $e \in E(G)$,
\begin{equation}\label{eqQ}\sum_{j \in e} \alpha_j=\sigma_e m+\frac{m}{2},\end{equation}
for some integer $\sigma_e$ associated with $e$.
\end{lem}

Let $G$ be an $m$-uniform connected hypergraph on $n$ vertices.
Consider the projective eigenvarieties of $\L(G)$ and $\Q(G)$ associated with zero eigenvalues:
\begin{align}\label{v0lq}
\PV_0^\L(G)&=\{\x \in \mathbb{P}^{n-1}: \L(G)\x^{m-1}=0\},\\
\PV_0^\Q(G)&=\{\x \in \mathbb{P}^{n-1}: \Q(G)\x^{m-1}=0\}.
\end{align}

By Lemma \ref{L&sL}, if $\x$ is a first Laplacian or signless Laplacian eigenvector of $G$, then $x_i \ne 0$ for each $i \in [n]$.
So we may assume that $x_1=1$ for each $\x$ of $\PV_0^\L(G)$ or $\PV_0^\Q(G)$ such that
\begin{align}\label{normalize}
\x=(e^{\i \frac{2\pi}{m} \alpha_1},\ldots,e^{\i \frac{2\pi}{m} \alpha_n}),
\end{align}
where $\alpha_1=0$ and $\alpha_i \in \Z_m$ for $i \in [n]$.

Denote
\begin{align}\label{s0lq}
\S_0^\L(G)&=\{\y \in \Z_m^n: B_G\y=0\},\\
\S_0^\Q(G)&=\{\y \in \Z_m^n: B_G\y=\frac{m}{2}\mathbf{1}\}.
\end{align}
Let $\alpha_{\x}=(\alpha_1,\ldots,\alpha_n)$.
Then, recalling the definition of the incidence matrix $B_G$ of $G$,  $\alpha_{\x} \in \S_0^\L(G)$ from Eq. (\ref{eqL}) if $\x$ is a first Laplacian eigenvector,
or $\alpha_{\x} \in \S_0^\Q(G)$ from Eq. (\ref{eqQ}) if $\x$ is a first signless Laplacian eigenvector.

Observe that $\S_0^\L(G)$ is a $\Z_m$-submodule of $\Z_m^n$, and
$$\S_0^\Q(G)=\bar\y+\S_0^\L(G)=\{\bar\y+\y: \y\in \S_0^\L(G)\}$$ if there exists a $\bar\y \in \S_0^\Q(G)$.
Note that $B_G \mathbf{1}=0$ over $\Z_m$ since the sum of each row of $B_G$ is equal to $m$.
So, if $\y \in \S_0^\L(G)$ (or $\S_0^\Q(G)$), then $\y+t \mathbf{1} \in \S_0^\L(G)$ (or $\S_0^\Q(G)$) for any $t \in \Z_m$.
It suffices to consider
\begin{align}\label{Ps0lq}
\PS_0^\L(G)&=\{\y \in \Z_m^n: B_G\y=0,y_1=0\},\\
\PS_0^\Q(G)&=\{\y \in \Z_m^n: B_G\y=\frac{m}{2}\mathbf{1},y_1=0\}.
\end{align}
Now $\PS_0^\L(G)$ is a $\Z_m$-submodule of $\S_0^\L(G)$;
   in fact, it is isomorphic to the quotient module $\S_0^\L(G)/(\Z_m\mathbf{1})$.
We also have $\PS_0^\Q(G)=\bar\y+\PS_0^\L(G)$ if there exists a  $\bar\y \in \PS_0^\Q(G)$.

\begin{lem}\label{bij}
Let $G$ be an $m$-uniform connected hypergraph on $n$ vertices.
Then there is a bijection between $\PV_0^\L(G)$ and $\PS_0^\L(G)$, and a bijection between $\PV_0^\Q(G)$ and $\PS_0^\Q(G)$.
\end{lem}

\begin{proof}
For each $\x \in \PV_0^\L(G)$, by Lemma \ref{L&sL}(i), we may assume that $\x$ has the form of (\ref{normalize}),
 i.e. $x_i=e^{\i \frac{2\pi}{m} \alpha_i}$ for some $\alpha_i \in \Z_m$ and $i \in [n]$, where $\alpha_1=0$.
  By Eq. (\ref{eqL}), $\alpha_{\x}:=(\alpha_1,\ldots,\alpha_n) \in \PS_0^\L(G)$.
Define a map
\begin{align}\label{mapL}
\Phi: \PV_0^\L(G) \to \PS_0^\L(G),  \x \mapsto \alpha_{\x}.
\end{align}
Obviously, $\Phi$ is an injective map.
On the other hand, for each $\alpha=(\alpha_1,\ldots,\alpha_n) \in \S_0^\L(G)$, as $B_G \alpha=0$ over $\Z_m$, for each edge $e \in E(G)$,
\[ \sum_{j \in e} \alpha_j \equiv 0 \mod m.\]
Let $\x_{\alpha}:=(e^{\i \frac{2\pi}{m} \alpha_1},\ldots,e^{\i \frac{2\pi}{m} \alpha_n})$.
Then it is easy to verify $\L(G) \x_{\alpha}^{m-1}=0$, implying $\x_{\alpha} \in \PV_0^\L(G)$ and $\Phi(\x_{\alpha})=\alpha$.
So $\Phi$ is a bijection between $\PV_0^\L(G)$ and $\PS_0^\L(G)$.

Similarly, for each $\x \in \PV_0^\Q(G)$, by Lemma \ref{L&sL}(ii),
 we may assume that $x_i=e^{\i \frac{2\pi}{m} \alpha_i}$ for some $\alpha_i \in \Z_m$ and $i \in [n]$, where $\alpha_1=0$.
  By Eq. (\ref{eqQ}), $\alpha_{\x}:=(\alpha_1,\ldots,\alpha_n) \in \PS_0^\Q(G)$.
 Define a map
\begin{align}\label{mapQ}
\Psi: \PV_0^\Q(G) \to \PS_0^\Q(G),  \x \mapsto \alpha_{\x}.
\end{align}
Then $\Psi$ is a bijection between $\PV_0^\Q(G)$ and $\PS_0^\Q(G)$.
\end{proof}

As $\PS_0^\L(G)$ is a $\Z_m$-modules, we now impose a Hadamard product $\circ$ on $\PV_0^\L(G)$ such that it is also $\Z_m$-modules, that is
\[ (\x \circ \y)_i =x_i y_i, i \in [n]. \]
Then for any $\x,\y \in \PV_0^\L(G)$, keeping in mind that $\x,\y$ have the form of (\ref{normalize}), by Lemma \ref{bij},
\[\Phi^{-1}(\Phi(\x)+\Phi(\y))=\x \circ \y \in \PV_0^\L(G),\]
where $\Phi$ is defined as in (\ref{mapL}).
So $(\PV_0^\L(G),\circ)$ is an abelian group, where $\mathbf{1}$ is the identity.
Note that $\x^{\circ m}=\x^{[m]}=\mathbf{1}$, implying that $\PV_0^\L(G)$ is a $\Z_m$-module isomorphic to $\PS_0^\L(G)$.

\begin{cor}\label{iso}
Let $G$ be an $m$-uniform connected hypergraph on $n$ vertices.
Then there is an isomorphism between $\Z_m$-modules $\PV_0^\L(G)$ and $\PS_0^\L(G)$.
\end{cor}

\begin{thm}\label{noL}
Let $G$ be an $m$-uniform connected hypergraph on $n$ vertices.
Suppose the incidence matrix $B_G$ has a Smith normal form over $\Z_m$ as in (\ref{smith}).
Then $1 \le r \le n-1$, and
\begin{align}\label{Lstru}
\PV_0^\L(G) \cong \PS_0^\L(G) \cong \oplus_{i, d_i \ne 1} \Z_{d_i} \oplus \underbrace{\Z_m \oplus \cdots \oplus \Z_m}_{n-1-r \text{\em \small~copies~}}.
\end{align}

Consequently, $G$ has $m^{n-1-r} \Pi_{i=1}^r d_i$ first Laplacian eigenvectors,
and the composition length of the $\Z_m$-module $\PV_0^\L(G)$ is
$\sum_{i, d_i \ne 1} \cl(d_i)+ (n-1-r)\cl(m)$.
\end{thm}

\begin{proof}
As $G$ is connected, $B_G \ne 0$, implying that $r \ge 1$.
Also, as $B_G \mathbf{1} =0$, adding all other columns to the last column of $B_G$, the last column of the resulting matrix becomes a zero column,
   which keeps invariant under any elementary row transformations and column transformations.
So $PB_G Q$ always has a zero column, implying that $r \le n-1$.

Let $\S'_0=\{\y \in \Z_m^n: P B_G Q \y=0\}$.
Then $\y \in \S_0^\L(G)$ if and only if $Q^{-1} \y \in \S'_0$, implying that $\S_0^\L(G)$ is $\Z_m$-module isomorphic to $\S'_0$.
It is easy to see
$$\S'_0 \cong \oplus_{i, d_i \ne 1} \Z_{d_i} \oplus \underbrace{\Z_m \oplus \cdots \oplus \Z_m}_{n-r \text{\small~copies~}}.$$
So by  Corollary \ref{iso} and the discussion in Section 2.5,
$$\PV_0^\L(G) \cong \PS_0^\L(G) \cong \S_0^\L(G)/(\Z_m \mathbf{1}) \cong \S'_0/(\Z_m (Q^{-1}\mathbf{1})) \cong
\oplus_{i, d_i \ne 1} \Z_{d_i} \oplus \underbrace{\Z_m \oplus \cdots \oplus \Z_m}_{n-1-r \text{\small~copies~}}.$$
\end{proof}

\begin{cor}\label{noS}
Let $G$ be an $m$-uniform connected hypergraph on $n$ vertices.
Suppose that $\Q(G)$ has a zero eigenvalue, and the incidence matrix $B_G$ has a Smith normal form over $\Z_m$ as in (\ref{smith}).
Then $G$ has $m^{n-1-r} \Pi_{i=1}^r d_i$ first signless Laplacian eigenvectors, which is equal to the number of first Laplacian eigenvectors.
\end{cor}

\begin{proof}
If $\Q(G)$ has a zero eigenvalue, then $\PV_0^Q(G)$ and hence $\PS_0^Q(G)$ is nonempty by Lemma \ref{bij}.
Let $\bar\y \in \PS_0^Q(G)$.
By the discussion prior to Lemma \ref{bij}, $\PS_0^\Q(G)=\bar\y+\PS_0^\L(G)$,
implying that $\PS_0^\Q(G)$ has the same cardinality as $\PS_0^\L(G)$.
The result follows from Theorem \ref{noL}.
\end{proof}

Finally we discuss when $\Q(G)$ has a zero eigenvalue. In fact the answer was given by Lemma \ref{L&sL}(ii).
Here we use another language arising from hypergraphs.

\begin{defi}\cite{Ni}
Let $G$ be an $m$-uniform hypergraph on $n$ vertices, where $m$ is even.
The hypergraph $G$ is called \emph{odd-colorable} if there exists a map
$f: V(G) \to [m]$ such that for each $e \in E(G)$,
\[ \sum_{v \in e} f(v)\equiv \frac{m}{2} \mod m.\]
The map $f$ is called an \emph{odd-coloring} of $G$.
\end{defi}

So, by Lemma \ref{L&sL}(ii), we get the following corollary immediately.

\begin{cor}\label{0-odd}
Let $G$ be an $m$-uniform connected hypergraph on $n$ vertices.
Then $\Q(G)$ has a zero eigenvalue if and only if $m$ is even and $G$ is odd-colorable.
\end{cor}

We also have some equivalent characterizations for a uniform connected hypergraph having zero signless Laplacian eigenvalue
by combining the results or using the techniques in \cite[Theorem]{SSW}, \cite[Theorem 18, Corollary 19]{Ni} and \cite[Theorem 3.2]{FKT}.

\begin{thm}\label{char-odd}
Let $G$ be an $m$-uniform connected hypergraph on $n$ vertices.
Then the following are equivalent.

\begin{enumerate}

\item  $m$ is even and $G$ is odd-colorable.

\item  $0$ is an eigenvalue of $\Q(G)$.

\item  The linear equation $B_G\y=\frac{m}{2}\mathbf{1}$ has a solution over $\Z_m$.

\item  $\Q(G)=D^{-(m-1)}\L(G)D$ for some diagonal matrix $D$ with $|D|=\I$.

\item  $\Spec(\L(G))=\Spec(\Q(G))$.

\item  $\rho(\L(G))=\rho(\Q(G))$.

\item  $\A(G)=-D^{-(m-1)}\A(G)D$ for some diagonal matrix $D$ with $|D|=\I$.

\item  $\Spec(\A(G))=-\Spec(\A(G))$.

\item  $-\rho(\A(G))$ is an eigenvalue of $\A(G)$.
\end{enumerate}
\end{thm}

\begin{proof}
We have known $(1) \Leftrightarrow (2)\Leftrightarrow (3)$ by Corollary \ref{0-odd} and Lemma \ref{bij}(ii), and $(4) \Rightarrow (5) \Rightarrow (6)$, $(7) \Rightarrow (8) \Rightarrow (9)$ as two diagonal similar tensors have the same spectra \cite{Shao}.
By the definition (\ref{product}), $(4) \Leftrightarrow (7)$. By Corollary 19 of \cite{Ni}, $(1) \Leftrightarrow (8)$.
So it suffices to prove $(6) \Rightarrow (4)$ and $(9) \Rightarrow (7)$.

If $\rho(\L(G))=\rho(\Q(G))$, taking $\la=\rho(\Q(G))e^{\i \theta}$ as an eigenvalue of $\L(G)$,
by Theorem \ref{PF2}(2), there exists a diagonal matrix $D$ with $|D|=\mathcal{I}$  such that
$$\Q(G)=e^{-\i \theta} D^{-(m-1)}\L(G)D.$$
So, $e^{-\i \theta}=1$ by comparing the diagonal entries of both sides, and hence
$\Q(G)=D^{-(m-1)}\L(G)D$.

If $-\rho(\A(G))$ is an eigenvalue of $\A(G)$, also using Theorem \ref{PF2}(2) by taking $\B=\A$,
we have a diagonal matrix $D$ with $|D|=\mathcal{I}$  such that
$\A(G)=-D^{-(m-1)}\A(G)D$.
\end{proof}

\section{First Laplacian or signless Laplacian H-eigenvectors}
Let $G$ be an $m$-uniform connected hypergraph on $n$ vertices.
Let $\x$ be a first Laplacian or signless Laplacian eigenvector of $G$.
By Lemma \ref{L&sL}, recalling that $\x$ has the form as in (\ref{normalize}), i.e.
\begin{align*}
\x=(e^{\i \frac{2\pi}{m} \alpha_1},\ldots,e^{\i \frac{2\pi}{m} \alpha_n}),
\end{align*}
where $\alpha_1=0$ and $\alpha_i \in \Z_m$ for $i \in [n]$.
If $\x$ is an H-eigenvector, then $x_i \in \{1,-1\}$, i.e. $\alpha_i \in \{0, \frac{m}{2}\}$ for $i \in [n]$,
where $x_1=1$ as $\alpha_1=0$.
We easily get the following result by the above discussion and Corollary \ref{0-odd},
which has been proved by Hu and Qi \cite{HQ}.

\begin{thm}\cite{HQ}\label{modd}
Let $G$ be an $m$-uniform hypergraph on $n$ vertices.
If $m$ is odd, then $G$ has the $\mathbf{1}$ as the only first Laplacian H-eigenvector, and has no first signless Laplacian H-eigenvectors.
\end{thm}

Suppose $m$ is even in the following.
So, if $\x$ is a first Laplacian H-eigenvector, letting $\beta_i=\alpha_i /(m/2)$ for $i \in [n]$, we have
\begin{align}\label{Hnorm}
\x=((-1)^{\beta_1}, \ldots, (-1)^{\beta_n}),
\end{align}
where $\beta_1=0$ and $\beta_i \in \Z_2$ for $i \in [n]$.
Dividing the Eq. (\ref{eqL}) by $m/2$, for each edge $e \in E(G)$,
\begin{align}\label{HeqL} \sum_{j \in e}\beta_j \equiv 0 \mod 2.
\end{align}

Similarly, if $\x$ is a first signless Laplacian H-eigenvector,
then $\x$ also has the form as in (\ref{Hnorm}), and by  Eq. (\ref{eqQ}) for each edge $e \in E(G)$,
\begin{align}\label{HeqQ} \sum_{j \in e}\beta_j \equiv 1 \mod 2.
\end{align}

We now consider the following two sets:
\begin{align}\label{Ps0lqH}
\HPS_0^\L(G)&=\{\y \in \Z_2^n: B_G\y=0,y_1=0\},\\
\HPS_0^\Q(G)&=\{\y \in \Z_2^n: B_G\y=\mathbf{1},y_1=0\}.
\end{align}
As $\Z_2$ is a field, $\HPS_0^\L(G)$ is a linear space over $\Z_2$.
If $\bar\y \in \HPS_0^\Q(G)$, then $\HPS_0^\Q(G)=\bar\y+\HPS_0^\L(G)$, a affine space over $\Z_2$.
Denote
\begin{align}\label{Hv0lq}
\HPV_0^\L(G)&=\{\x \in \mathbb{P}^{n-1}: \L(G)\x^{m-1}=0, \x \mbox{~is an H-eigenvector}\},\\
\HPV_0^\Q(G)&=\{\x \in \mathbb{P}^{n-1}: \Q(G)\x^{m-1}=0, \x \mbox{~is an H-eigenvector}\}.
\end{align}

\begin{lem}\label{Hbij}
Let $G$ be an $m$-uniform connected hypergraph on $n$ vertices.
Then there is a bijection between $\HPV_0^\L(G)$ and $\HPS_0^\L(G)$, and a bijection between $\HPV_0^\Q(G)$ and $\HPS_0^\Q(G)$.
\end{lem}

\begin{proof}
For each $\x \in \HPV_0^\L(G)$,  $\x$ has the form as in (\ref{Hnorm}),
 i.e. $\x=((-1)^{\beta_1}, \ldots, (-1)^{\beta_n})$, where $\beta_1=0$ and $\beta_i \in \Z_2$ for $i \in [n]$.
  By Eq. (\ref{HeqL}), $\beta_{\x}:=(\beta_1,\ldots,\beta_n) \in \HPS_0^\L(G)$.
Define a map
\begin{align}\label{HmapL}
\Phi_H: \HPV_0^\L(G) \to \HPS_0^\L(G),  \x \mapsto \beta_{\x}.
\end{align}
Obviously, $\Phi_H$ is an injective map.
On the other hand, for each $\beta=(\beta_1,\ldots,\beta_n) \in H\PS_0^\L(G)$, as $B_G \beta=0$ over $\Z_2$, for each edge $e \in E(G)$,
\[ \sum_{j \in e} \beta_j \equiv 0 \mod 2.\]
Let $\x_{\beta}:=((-1)^{\beta_1}, \ldots, (-1)^{\beta_n})$.
Then it is easy to verify $\L(G) \x_{\beta}^{m-1}=0$, implying $\x_{\beta} \in \HPV_0^\L(G)$ and $\Phi_H(\x_{\beta})=\beta$.
So $\Phi_H$ is a bijection between $\HPV_0^\L(G)$ and $\HPS_0^\L(G)$.

Similarly, for each $\x \in \HPV_0^\Q(G)$,
 $\x=((-1)^{\beta_1}, \ldots, (-1)^{\beta_n})$, where $\beta_1=0$ and $\beta_i \in \Z_2$ for $i \in [n]$.
  By Eq. (\ref{HeqQ}), $\beta_{\x}:=(\beta_1,\ldots,\beta_n) \in \HPS_0^\Q(G)$.
 Define a map
\begin{align}\label{HmapQ}
\Psi: \HPV_0^\Q(G) \to \HPS_0^\Q(G),  \x \mapsto \beta_{\x}.
\end{align}
Then $\Psi$ is a bijection between $\HPV_0^\Q(G)$ and $\HPS_0^\Q(G)$.
\end{proof}

    Define a Hadamard product $\circ$ in $\HPV_0^\L(G)$.
    Then by Lemma \ref{Hbij} and a discussion similar to that prior to Corollary \ref{iso}, we get that $(\HPV_0^\L(G),\circ)$ is a  $\Z_2$-linear space.

\begin{cor}\label{Hiso}
Let $G$ be an $m$-uniform connected hypergraph on $n$ vertices.
Then there is an isomorphism between $\Z_2$-linear spaces $\HPV_0^\L(G)$ and $\HPS_0^\L(G)$.
\end{cor}

\begin{thm}\label{HnoL}
Let $G$ be an $m$-uniform connected hypergraph on $n$ vertices, where $m$ is even.
Suppose the incidence matrix $B_G$ has rank $\bar{r}$ over $\Z_2$.
Then $1 \le \bar r \le n-1$, and $\HPV_0^\L(G)$ is a $\Z_2$-linear spaces of dimension $n-1-\bar{r}$.
Consequently, $G$ has $2^{n-1-\bar{r}}$ first Laplacian H-eigenvectors.
\end{thm}

\begin{proof}
Let ${\rm H}\S_0^\L(G)=\{\y \in \Z_2^n: B_G\y=0\}$.
Then ${\rm H}\S_0^\L(G)$ is a $\Z_2$-linear space of dimension $n - \bar{r}$.
As $m$ is even, $B_G \mathbf{1}=0$ over $\Z_2$.
So, $\HPS_0^\L(G)$ is isomorphic to ${\rm H}\S_0^\L(G)/(\Z_1 \mathbf{1})$, which is a $\Z_2$-linear space of dimension $n - 1-\bar{r}$.
The result follows by Corollary \ref{Hiso}.
\end{proof}

\begin{cor}\label{HnoS}
Let $G$ be an $m$-uniform connected hypergraph on $n$ vertices, where $m$ is even.
Suppose that $\Q(G)$ has a zero $H$-eigenvalue, and the incidence matrix $B_G$ has  rank $\bar{r}$  over $\Z_2$.
Then $G$ has $2^{n-1-\bar{r}}$ first signless Laplacian H-eigenvectors, which is equal to the number of first Laplacian H-eigenvectors.
\end{cor}

\begin{proof}
If $\Q(G)$ has a zero H-eigenvalue, then $\HPV_0^Q(G)$ and hence $\HPS_0^Q(G)$ is nonempty by Lemma \ref{Hbij}.
Let $\bar\y \in \HPS_0^Q(G)$.
By the discussion prior to Lemma \ref{Hbij}, $\HPS_0^\Q(G)=\bar\y+\HPS_0^\L(G)$.
The result follows from Theorem \ref{HnoL}.
\end{proof}

Finally we discuss when $\Q(G)$ has a zero H-eigenvalue. In fact, the characterization was given by Hu and Qi \cite[Proposition 5.1]{HQ}, Shao et al. \cite[Theorem 2.5]{SSW}.
We need the following notions: odd (even) traversal and odd (even) bipartition.
Let $G$ be hypergraph.
A set $U$ of vertices of $G$ is called an \emph{odd (even) transversal} if every edge of $G$ intersects $U$ in an odd (even) number of vertices \cite{CH,RS,Ni}.
$G$ is called \emph{odd (even)-traversal} if $G$ has an odd (even) transversal.
Suppose further $G$ is an $m$-uniform hypergraph, where $m$ is even.
Then the odd (even) transversal has another statement.
An \emph{odd (even) bipartition} $\{V_1,V_2\}$ of $G$ is a bipartition of $V(G)$
such that each edge of $G$ intersects $V_1$ or $V_2$ in an odd (even) number of vertices.
$G$ is called \emph{odd (even)-bipartite} if $G$ has an odd (even) bipartition \cite{HQ}.
Here we allow a trivial case on the even bipartitions, that is, $V_1$ or $V_2$ may be $V(G)$, which is not included in definition given in \cite{HQ}.

We have the following equivalent characterizations on a uniform connected hypergraph having a zero H-eigenvalue
by combining the results in \cite{HQ,SSW,Ni} or considering Theorem \ref{char-odd} in a special case.

\begin{thm}\label{Hchar-odd}
Let $G$ be an $m$-uniform connected hypergraph on $n$ vertices.
Then the following are equivalent.

\begin{enumerate}

\item $m$ is even, and $G$ is odd-bipartite or odd-transversal.

\item $0$ is an H-eigenvalue of $\Q(G)$.

\item The linear equation $B_G\y=\mathbf{1}$ has a solution over $\Z_2$.

\item $\Q(G)=D^{-(m-1)}\L(G)D$ for some diagonal matrix $D$ with $\pm 1$ on the diagonal.

\item $\HSpec(\L(G))=\HSpec(\Q(G))$.

\item $\rho(\L(G))=\rho(\Q(G))$, and $\rho(\Q(G))$ is an H-eigenvalue of $\L(G)$.

\item $\A(G)=-D^{-(m-1)}\A(G)D$ for some diagonal matrix $D$ with $\pm 1$ on the diagonal.

\item $\HSpec(\A(G))=-\HSpec(\A(G))$.

\item $-\rho(\A(G))$ is an H-eigenvalue of $\A(G)$.
\end{enumerate}
\end{thm}

By Theorem \ref{char-odd} and Theorem \ref{Hchar-odd}, we know the difference between odd-colorable and odd-bipartite (or odd-transversal).
An odd-bipartite hypergraph is odd-colorable, but the converse does not hold.
Nikiforov \cite{Ni} proved that if $m \equiv 2 \mod 4$, then an $m$-uniform odd-colorable hypergraph is odd-bipartite.
For $m \equiv 0 \mod 4$, Nikiforov \cite{Ni} constructed two families of $m$-uniform odd-colorable hypergraphs which are not odd-bipartite.

In the paper \cite{KF}, the authors construct a family of $m$-uniform hypergraphs from simple graphs,
called the {\it generalized power hypergraph} $G^{m,m/2}$, which is obtained from a simple graph
$G$ by blowing up each vertex into an $m/2$-set and preserving the adjacency relation, where $m$ is even.
It was shown that  $G^{m,m/2}$ is odd-bipartite if and only if $G$ is bipartite \cite{KF}.
Suppose now that $G$ is connected and non-bipartite.
It was shown that $\rho(\L(G^{m,m/2}))=\rho(\Q(G^{m,m/2}))$ if and only if $m \equiv 0 \mod 4$ \cite{FKT}.
So, by Theorem \ref{char-odd}, $G^{m,m/2}$ is odd-colorable but is not odd-bipartite if $m \equiv 0 \mod 4$, and $G^{m,m/2}$ is not odd-colorable otherwise.

In the paper \cite{HQ}, the authors counted the number of first Laplacian or signless Laplacian H-eigenvectors of
a general uniform (not necessarily connected) hypergraph by means of the odd (even)-bipartite connected components.
They also discuss the number of first Laplacian or signless Laplacian N-eigenvectors for $3$, $4$ or $5$-uniform hypergraphs by means of some kinds partitions.
However, it is not easy to determine the number of even (odd)-bipartite connected components or the partitions as described in the paper.

We revisit this problem by means of incidence matrix of a hypergraph.
Suppose $G$ is an $m$-uniform hypergraph. If $G$ is disconnected, it suffices to consider its connected components.
If $m$ is odd and $G$ is connected, by Theorem \ref{modd}, $G$ has the $\mathbf{1}$ as the only first Laplacian eigenvector, and has no first signless Laplacian eigenvectors.
So we assume $G$ is an $m$-uniform connected hypergraph with vertex set $[n]$, where $m$ is even.
By Lemma \ref{Hbij}, the number of first Laplacian (signless Laplacian) H-eigenvectors of $G$ is exactly that of
the solutions to the equation $B_G \y = 0$ ($B_G \y = \mathbf{1}$) over $\Z_2$ with $y_1=0$.
Each solution $\y$ to $B_G \y = 0$ ($B_G \y = \mathbf{1}$) over $\Z_2$  with $y_1=0$ determines uniquely an even (odd) bipartition $\{V_0,V_1\}$ of $G$, where
$$ V_i=\{v: y_v=i\}, i \in \{0,1\}.$$
On the other hand, each even (odd) bipartition $\{V_0,V_1\}$ of $G$ by fixing $1 \in V_0$
   gives uniquely a solution $\y$ to $B_G \y = 0$ ($B_G \y = \mathbf{1}$) over $\Z_2$ with $y_1=0$, by setting $y_v=i$ if $v \in V_i$ for $i \in \{0,1\}$.
So we get following result on the number of even (odd) bipartitions of $G$ by the above discussion,
 Theorem \ref{HnoL} and Corollary \ref{HnoS}

\begin{thm}\label{bipa}
Let $G$ be an $m$-uniform connected hypergraph, where $m$ is even.
Then the number of even (odd) bipartitions of $G$ is exactly the number of solutions to $B_G \y = 0$ ($B_G \y = \mathbf{1}$) over $\Z_2$ with $y_1=0$, which is equal to the number of first Laplacian (signless Laplacian) H-eigenvectors of $G$.

The number of even bipartitions of $G$ is $1$ (corresponding to the trivial bipartition) or a power of $2$,
which is equal to the number of odd bipartitions of $G$ if $G$ has odd bipartitions.
\end{thm}

The number of first Laplacian (signless Laplacian) N-eigenvectors of a connected hypergraph $G$ is the number of first Laplacian (signless Laplacian) eigenvectors minus the number of first Laplacian (signless Laplacian) H-eigenvectors, which can be obtained explicitly by Lemma \ref{modd}, Theorem \ref{noL} and Theorem \ref{HnoL} (Corollary \ref{noS} and Corollary \ref{HnoS}).

\begin{thm}\label{last}
Let $G$ be an $m$-uniform connected hypergraph, and let $B_G$ has a Smith normal form over $\Z_m$ as in (\ref{smith}).

\begin{enumerate}
\item  If $m$ is odd, the number of first Laplacian (or first signless Laplacian) N-eigenvectors of $G$ is
$m^{n-1-r} \Pi_{i=1}^r d_i-1$ (or $0$).

\item  If $m$ is even, the number of first Laplacian N-eigenvectors of $G$ is
$$m^{n-1-r} \Pi_{i=1}^r d_i-2^{n-1-\bar{r}},$$
which is equal to the number of first signless Laplacian N-eigenvectors of $G$ if zero is an H-eigenvalue of $\Q(G)$,
where $\bar{r}$ is the rank of $B_G$ over $\Z_2$.

\item  If $m$ is even, the number of first signless Laplacian N-eigenvectors of $G$ is
$$m^{n-1-r} \Pi_{i=1}^r d_i,$$
if zero is an N-eigenvalue of $\Q(G)$.
\end{enumerate}
\end{thm}

Note that if $B_G$ has a Smith normal form over $\Z_m$ as in (\ref{smith}) and $m$ is even,
then an invertible element $t$ of $\Z_m$ is coprime to $m$, which is necessarily an odd number as $m$ is even.
So, such $t$ is also invertible in $\Z_2$, and all elementary transformations over $\Z_m$ are also  elementary transformations over $\Z_2$.
So the Smith normal form of $B_G$ over $\Z_2$ is the Smith normal form (\ref{smith}) modulo $2$, and hence the rank of $B_G$  over $\Z_2$ is exactly the number of odd invariant divisors of $B_G$ over $\Z_m$.

Finally we will give some examples of Theorem \ref{last}.

\begin{exm}
Let $G$ be a complete $m$-uniform hypergraph on $n$ vertices, where $n \ge m+1$.
For any two vertices $i \ne j$, taking an arbitrary $(m-1)$-set $U$ of $G$ which does not contain $i$ or $j$,
by considering the equation $B_G\y =0$ on two edges $U \cup \{i\}$ and $U \cup \{j\}$,
we have $y_i=y_j$.
So the equation only has the solutions $\alpha \mathbf{1}$ for some $\alpha \in \Z_m$.
By Lemma \ref{bij}, $G$ has only one first Laplacian eigenvector associated with the zero eigenvalue, i.e. the H-eigenvector $\mathbf{1}$.

If $m$ is even, consider the equation $B_G\y =\frac{m}{2}\mathbf{1}$ over $\Z_m$.
By a similar discussion, we have $\y=\alpha \mathbf{1}$ for some $\alpha \in \Z_m$, which implies that the above equation has no solutions.
So, also by Lemma \ref{bij}, $G$ has no zero signless Laplacian eigenvalue, which implies that $G$ is not odd-colorable by Theorem \ref{char-odd}.
\end{exm}

A \emph{cored hypergraph} \cite{HQS} is one such that each edge contains a vertex of degree one.
It is easily seen that a cored hypergraph of even uniformity is odd-bipartite.

\begin{exm}
Let $G$ be a connected $m$-uniform cored hypergraph on $n$ vertices with $t$ edges $e_1,e_2,\ldots,e_t$.
Choose one vertex $i$ from each edge $e_i$ for $i \in [t]$.
Then the incidence matrix $B_G$ contains an $t \times t$ identity submatrix, implying that $B_G$ has $t$ invariant divisors all being $1$ over $\Z_m$ or $\Z_2$.
So, by Theorems \ref{noL} and \ref{HnoL}, $G$ has $m^{n-1-t}$ first Laplacian eigenvectors and $2^{n-1-t}$ first Laplacian H-eigenvectors.

Suppose that $m$ is even.
Then $G$ is odd-bipartite, and zero is an H-eigenvalue of signless Laplacian of $G$ by Theorem \ref{Hchar-odd}.
So, by Corollaries \ref{noS} and \ref{HnoS}, $G$ has $m^{n-1-t}$ first signless Laplacian eigenvectors and $2^{n-1-t}$ first signless Laplacian H-eigenvectors.
\end{exm}

\begin{exm}
Let $G^{m,m/2}$ be a generalized power hypergraph, where $G$ is connected and non-bipartite, and $m$ is a multiple of $4$.
By the discussion after Theorem \ref{Hchar-odd}, $G^{m,m/2}$ is a non-odd-bipartite but odd-colorable hypergraph.
So, zero is an N-eigenvalue of signless Laplacian of $G$.

In particular, take $G=C_3^{4,2}$, where $C_3$ is a triangle as a simple graph.
Then by solving the Smith normal form of the incidence matrix of $C_3^{4,2}$,
it has invariant divisors $1,1,2$ over $\Z_4$ and invariant divisors $1,1$ over $\Z_2$.
So, $C_3^{4,2}$ has $32$ first Laplacian eigenvectors and $8$ first signless Laplacian H-eigenvectors.
As zero is an N-eigenvalue of signless Laplacian of $C_3^{4,2}$, $C_3^{4,2}$ has $32$ first signless Laplacian eigenvectors, all being N-eigenvectors.
\end{exm}


\begin{thebibliography}{99}

\bibitem{AF} F. W. Andersion, K. R. Fuller, \emph{Rings and Categories of Modules}, Springer-Verlag, New York, 1992.

\bibitem{Ber} C. Berge, \emph{Hypergraphs: Combinatorics of finite sets}, North-Holland, 1989.

\bibitem{CPZ} K. C. Chang, K. Pearson, T. Zhang, Perron-Frobenius theorem for nonnegative tensors, \emph{Commu. Math. Sci.}, \textbf{6}(2008), 507-520.

\bibitem{CPZ2} K. C. Chang, K. Pearson, T. Zhang, On eigenvalue problems of real symmetric tensors,
\emph{ J. Math. Anal. Appl.}, \textbf{350}(2009), 416-422.

\bibitem{CD} J. Cooper, A. Dutle, Spectra of uniform hypergraphs, \emph{Linear Algebra Appl.}, \textbf{436}(9)(2012), 3268-3292.

\bibitem{CH} R. Cowen, S. H. Hechler, J. W. Kennedy, A. Steinberg, Odd neighborhood transversals on grid graphs, \emph{Discrete Math.}, \textbf{307}(2007), 2200-2208.


\bibitem{CQ} H. Chen, L. Qi, Some spectral properties of odd-bipartite $Z$-tensors and their absolute tensors,
 \emph{Front. Math. China}, \textbf{11}(3)2016, 539-556.

\bibitem{FKT} Y.-Z. Fan, M. Khan, Y.-Y. Tan, The largest H-eigenvalue and spectral radius of Laplaican tensor of non-odd-bipartite generalized power hypergraphs,
\emph{Linear Algebra Appl.}, \textbf{504}(2016), 487-502.

\bibitem{FHB} Y.-Z. Fan, T. Huang, Y.-H. Bao, C.-L. Zhuan-Sun, Y.-P. Li,
The spectral symmetry of weakly irreducible nonnegative tensors and connected hypergraphs, \emph{Trans. Amer. Math. Soc.}, DOI: https://doi.org/10.1090/tran/7741.

\bibitem{FBH} Y.-Z. Fan, Y.-H. Bao, T. Huang,
Eigenvariety of nonnegative symmetric weakly irreducible tensors associated with spectral radius  and its application to hypergraphs, \emph{Linear Algebra Appl.}, \textbf{564} (2019), 72-94.

\bibitem{FGH} S. Friedland, S. Gaubert, L. Han, Perron-Frobenius theorem for nonnegative multilinear forms and extensions,
 \emph{Linear Algebra Appl.}, \textbf{438}(2013), 738-749.

\bibitem{Ha} R. Hartshorne, \textit{Algebraic Geometry}, Springer-Verlag, New York, 1977.

\bibitem{HQ} S. Hu, L. Qi, The eigenvectors associated with the zero eigenvalues of the Laplacian and signless Laplacian tensors of a uniform hypergraph,
\emph{Discrete Appl. Math.}, \textbf{169}(2014), 140-151.

\bibitem{HQS} S. Hu, L. Qi, J. Y. Shao, Cored hypergraphs, power hypergraphs and their Laplacian H-eigenvalues, \emph{ Linear Algebra Appl.}, \textbf{439}(2013), 2980-2998.

\bibitem{HQX} S. Hu, L. Qi, J. Xie, The largest Laplacian and signless Laplacian H-eigenvalues of a uniform hypergraph,  \emph{Linear Algebra Appl.}, \textbf{469}(2015), 1-27.

\bibitem{HuYe} S. Hu, K. Ye, Mulplicities of tensor eigenvalues, \emph{Commu. Math. Sci.}, \textbf{14}(2016), 1049-1071.

\bibitem{KLQY} L. Kang, L. Liu, L. Qi, X. Yuan, Spectral radii of two kinds of uniform hypergraphs, \emph{Appl. Math. Comput.}, \textbf{338} (2018) 661-668.

\bibitem{KF} M. Khan, Y.-Z. Fan, On the spectral radius of a class of non-odd-bipartite even uniform hypergraphs, \emph{Linear Algebra Appl.},  \textbf{480}(2015), 93-106.

\bibitem{KF2} M. Khan, Y.-Z. Fan, Y.-Y. Tan, The H-spectra of a class of generalized power hypergraphs, \emph{Discrete Math.}, \textbf{339}(2016), 1682-1689.


\bibitem{Lim} L.-H. Lim, Singular values and eigenvalues of tensors: a variational approach,
\emph{Proceedings of the 1st IEEE International Workshop on Computational Advances in Multi-Sensor Adaptive Processing}, 2005, pp. 129-132.

\bibitem{LM} L. Lu, S. Man, Connected hypergraphs with small spectral radius, \emph{Linear Algebra Appl.},  \textbf{509}(2016), 206-227.




\bibitem{Ni} V. Nikiforov, Hypergraphs and hypermatrices with symmetric spectrum,  \emph{Linear Algebra Appl.}, \textbf{519}(2017), 1-18.

\bibitem{PZ} K. Pearson, T. Zhang, On spectral hypergraph theory of the adjacency tensor, \emph{Graphs Combin.}, \textbf{30}(5) (2014), 1233-1248.

\bibitem{RS} D. Rautenbach, Z. Szigeti, Greedy colorings of words, \emph{Discrete Appl. Math.}, \textbf{160}(2012), 1872-1874.



\bibitem{Qi} L. Qi, Eigenvalues of a real supersymmetric tensor, \emph{J. Symbolic Comput.}, \textbf{40}(6)(2005), 1302-1324.

\bibitem{Qi2} L. Qi, Eigenvalues and invariants of tensors, \emph{J. Math. Anal. Appl.}, \textbf{325}(2007), 1363-1377.

\bibitem{Qi3} L. Qi, H$^+$-eigenvalues of Laplacian and signless Laplacian tensor, \emph{Commu. Math. Sci.}, \textbf{12}(2014), 1045-1064.


\bibitem{Shao} J. Y. Shao, A general product of tensors with applications, \emph{ Linear Algebra Appl.}, \textbf{439}(2013), 2350-2366.

\bibitem{SSW} J.-Y. Shao, H.-Y. Shan and B.-F. Wu, Some spectral properties and characterizations of
connected odd-bipartite uniform hypergraphs, \emph{Linear Multilinear Algebra}, \textbf{63}(2015), 2359-2372.

\bibitem{YY1} Y. Yang and Q. Yang,
Further results for Perron-Frobenius theorem for nonnegative tensors, \emph{SIAM J Matrix Anal. Appl.}, \textbf{31}(5)(2010), 2517-2530.

\bibitem{YY2} Y. Yang and Q. Yang,
Further results for Perron-Frobenius theorem for nonnegative tensors II, \emph{SIAM J Matrix Anal. Appl.}, \textbf{32}(4)(2011), 1236-1250.

\bibitem{YY3} Y. Yang, Q. Yang, On some properties of nonnegative weakly irreducible tensors, Available at arXiv: 1111.0713v2.

\bibitem{ZSWB} J. Zhou, L. Sun, W. Wang, C. Bu, Some spectral properties of uniform hypergraphs, \emph{Electron. J. Combin.}, \textbf{21}(4)(2014), \#P4.24.


\end{thebibliography}
\end{document}